\documentclass{amsart}
\usepackage{amsfonts,amssymb,amsmath,amsthm, fullpage}
\usepackage{url}
\usepackage{enumerate}

\newtheorem{thrm}{Theorem}[section]
\newtheorem{lem}[thrm]{Lemma}

\theoremstyle{definition}
\newtheorem{definition}[thrm]{Definition}
\newtheorem{remark}[thrm]{Remark}
\numberwithin{equation}{section}

\newcommand{\R}{\mathbb R}
\newcommand{\N}{\mathbb N}

\newcommand{\Z}{\mathbb Z}

\newcommand{\p}{\partial}

\newcommand{\backs}{\backslash}
\newcommand{\ds}{\displaystyle}

\author{Li-An Daniel Wang}
\address{
Department of Mathematics and Statistics \\
Sam Houston State University \\
Huntsville, TX }
\email{daniel.wang@shsu.edu}
\thanks{
The author was supported by the departmental funds from Sam Houston State University and the Johnson Fellowship from the University of Oregon. We would also like to thank M. Bownik for his patient guidance in this project. 
}

\keywords{Anisotropic Hardy spaces, Fourier transform, multipler theorem}
\subjclass[2010]{42B30 (42B25, 42B35)}
\begin{document}

\title[Multiplier Theorem on $H_A^p$]{A Multiplier Theorem on Anisotropic Hardy Spaces}

\begin{abstract} We present a multiplier theorem on anisotropic Hardy spaces. When $m$ satisfies the anisotropic, pointwise Mihlin condition, we obtain boundedness of the multiplier operator $T_m : H_A^p (\R^n) \rightarrow H_A^p (\R^n)$, for the range of $p$ that depends on the eccentricities of the dilation $A$ and the level of regularity of a multiplier symbol $m$. This extends the classical multiplier theorem of Taibleson and Weiss \cite{MR604370}.

\end{abstract}
\maketitle

\section{Introduction} \label{sect1}

We present a multiplier theorem (Theorem \ref{Thm-Mult1}) on anisotropic Hardy space $H_A^p (\R^n)$. This space was first studied by Bownik \cite{MR1982689}, and generalizes the classical Hardy space of Fefferman and Stein \cite{MR0447953} as well as the parabolic Hardy spaces of Calder\'{o}n and Torchinsky \cite{MR0417687} with a geometry and quasinorm induced by an expansive matrix $A$. Since the introduction of $H_A^p$, the anisotropic structure has been extended to a number of settings: Besov \cite{MR2179611} and Triebel-Lizorkin spaces \cite{MR2186983},  weighted anisotropic Hardy spaces \cite{MR2492226}, variable Hardy-Lorentz spaces \cite{ABR-2016}, and pointwise variable anisotropy \cite{MR2838119}, to name just a few. However, the study of the Fourier transform on these further generalizations are still incomplete, given that analysis of the Fourier transform becomes substantially harder. 

To state our multiplier theorem, we require a few definitions; more details are in Section 2.  Let $A$ be an $n \times n$ matrix, and $|\det A| = b$. We say $A$ is a dilation matrix if all eigenvalues $\lambda$ of $A$ satisfy $|\lambda| > 1$. If $\lambda_1, \ldots, \lambda_n$ are the eigenvalues of $A$, ordered by their norm from smallest to largest, then define $\lambda_-$ and $\lambda_+$ to satisfy $1 < \lambda_- < |\lambda_1|$ and $|\lambda_n| < \lambda_+$. Associated with $A$ is a sequence of nested ellipsoids $\{ B_j \}_{j \in \Z}$ such that $B_{j + 1} = A(B_j)$ and $|B_0| = 1$. If $A^{*}$ is the adjoint of $A$, then $A^{*}$ is also a dilation matrix with the same determinant $b$ and eigenvalues, with its own nested ellipsoids $\{ B_k^{\ast} \}_{k \in \Z}$.

We use $\hat{f}$ and $\check{f}$ to denote the Fourier and inverse Fourier transforms of $f$ respectively. We say a measurable function $m \in L^{\infty}$ is a multiplier on $H_A^p$ if its associated multiplier operator, initially defined by $T_m f = (\hat{f} m)^{\vee}$ for $f \in L^2 \cap H_A^p$, is bounded $H_A^p \rightarrow H_A^p$. We reserve $\xi$ for the independent variable in the frequency domain, and $\p_{\xi}$ denotes differentiation with respect to $\xi$. For a dilation matrix $A$, we define the dilation operator by $D_A f (x) = f(Ax)$. Henceforth, $C$ will denote a general constant which may depend on the dilation matrix $A$ and any scalar parameters $n, p, q$, and may change from line to line, but independent of $f \in H_A^p$. The regularity requirement of a multiplier $m$ will be given by the following Mihlin condition.

\begin{definition} Let $A$ be a dilation matrix. Let $N \in \N \cup \{ 0 \}$ and let $m \in C^N (\R^n \backs \{ 0 \})$. We say $m$ satisfies the anisotropic Mihlin condition of order $N$ if there exists a constant $C = C_N$ such that for all multi-indices $\beta$,  $|\beta| \leq N$, all $j \in \Z$, and all $\xi \in B_{j + 1}^{\ast} \backs B_j^{\ast}$,
    \begin{align}\label{Def-M}
    |D_{A^{\ast}}^{-j} \p_{\xi}^{\beta} D_{A^{\ast}}^j m(\xi)| \leq C.
    \end{align}
\end{definition}

We can now state our main result. For $r \in \R$, the (integer) floor of $r$ is given by $\lfloor r \rfloor$. 
    \begin{thrm}\label{Thm-Mult1} Let $A$ be a dilation matrix, $N \in \N$, and denote $L = \left(N\frac{\log \lambda_-}{\log b} - 1 \right) \frac{\log b}{\log \lambda_+}$. If $m$ satisfies the Mihlin condition of order $N$ and $T_m$ is the corresponding multiplier operator, then $T_m : H_A^p \rightarrow H_A^p$ is bounded provided $p$ satisfies
        \begin{align}\label{range:multiplier}
        0 \leq \frac{1}{p} - 1 < \left\lfloor L \right\rfloor \frac{(\log \lambda_-)^2}{\log b \log \lambda_+}.
        \end{align}
    \end{thrm}

\begin{remark}\label{def:tighten} We have implicitly fixed $\lambda_-$ and $\lambda_+$, determining the eccentricities of our dilation matrix. However, we can always `tighten' the eccentricities by defining $\tilde{\lambda}_-$ and $\tilde{\lambda}_+$ so that
	\[ 1 < \lambda_- < \tilde{\lambda}_- < |\lambda_1| \leq \ldots \leq |\lambda_n| < \tilde{\lambda}_+ < \lambda_+. \] 
In the proof of Theorem \ref{Thm-Mult1}, we will exploit this simple fact.

\end{remark}

\begin{remark} An instructive example for the dilation matrix is by setting $A = 2I_n$, so $\lambda_- = \lambda_+ = 2$ and $b = 2^n$. Then \eqref{range:multiplier} is equivalent to $\frac{n}{N - 1} < p \leq 1$, thus recovering the classical case. 

\end{remark}


As an essential class of singular integral operators, multiplier operators have been well studied for the classical Hardy space $H^p$ and its various extensions. We briefly discuss four classical multiplier theorems that are related to Theorem \ref{Thm-Mult1}.

First, our proof of Theorem \ref{Thm-Mult1} most closely resembles that of Peetre \cite{MR0380394} in that if $m$ satisfies a classical pointwise Mihlin condition with respect to the Euclidean norm (which condition \eqref{Def-M} generalizes), then it is a multiplier on Triebel-Lizorkin and Besov-Lipschitz spaces. Second, this pointwise Mihlin condition is stronger than an integral H\"{o}rmander condition on $m$, used in Taibleson and Weiss \cite{MR604370} and paired with molecular decomposition of $H^p$ to prove the boundedness of $T_m$. 
Third, this H\"{o}rmander condition is equivalent to a Herz-norm condition on the inverse-Fourier transforms of smooth truncations of the multiplier $m$, which Calder\'{o}n and Torchinsky \cite{MR0417687} used to prove the multiplier theorem in the parabolic setting. 
Lastly, Baernstein and Sawyer \cite{MR776176} further generalized this with a weaker Herz-norm condition, generalizing the previous three multiplier theorems. 

For our multiplier theorem, we will assume the (strongest) Mihlin condition on $m$ to overcome the issues native to the anisotropic setting. This approach was first considered by Benyi and Bownik \cite{MR2720206} in the study of symbols associated with pseudo-differential operators. Our Theorem \ref{Thm-Mult1} is closely related to their result, though we require minimal regularity requirement on $m$, and we obtain a more precise range of exponents $p$ for which multiplier operators are bounded in terms of eccentricity of the dilation $A$, as measured by $\frac{\log \lambda_-}{\log \lambda_+}$ and $\frac{\log_{\lambda_-}}{\log b}$. Ding and Lan \cite{MR2316761} extended the multiplier theorems of \cite{MR776176} to the spaces $T : H_A^p \rightarrow L^p$, though with an additional requirement that the dilation $A$ is symmetric.

The rest of the paper will be organized as follows. In Section 2, we give the background information on anisotropic Hardy spaces $H_A^p$. In Section 3, we give the lemmas needed for the proof of Theorem \ref{Thm-Mult1}, from which the theorem follows immediately. In Section 4, we provide the proofs of the lemmas as well as the molecular decomposition of $H_A^p$.

\section{Anisotropic Hardy spaces and Multiplier Operators} We now introduce the anisotropic structure and the associated Hardy spaces. Given a dilation matrix $A$, we can find a (non-unique) homogeneous quasi-norm, that is, a measurable mapping $\rho_A : \R^n \rightarrow [0, \infty)$ with a doubling constant $c$ satisfying:
        \begin{center}
        \begin{tabular}{lll}
        $\rho_A (x) = 0$
            &   exactly when    &   $x = 0$, \\
        $\rho_A (Ax) = b\rho(x)$
            &   for all         &   $x \in \R^n$, \\
        $\rho_A(x + y) \leq c(\rho_A(x) + \rho_A(y))$  &  for all & $x, y \in \R^n$. \\
        \end{tabular}
        \end{center}
Note that $(\R^n, dx, \rho_A)$ is a space of homogeneous type ($dx$ denotes the Lebesgue measure), and any two quasi-norms associated with $A$ will give the same anisotropic structure. In the isotropic setting, the `basic' geometric object is the Euclidean ball $B(x, r)$, centered at $x \in \R^n$ with radius $r$. This has the nice property that whenever $r_1 < r_2$, we have $B(x, r_1) \subset B(x, r_2)$. But for a dilation matrix $A$, we do not expect $B(x, r) \subset A(B(x, r))$. Instead, one can construct `canonical' ellipsoids $\{ B_k \}_{k \in \Z}$, associated with $A$, such that for all $k$, $B_{k + 1} = A(B_k)$, $B_k \subseteq B_{k + 1}$,  and $|B_k| = b^k$. These nested ellipsoids will serve as the basic geometric object in the anisotropic setting. Moreover, we can use the ellipsoids to define the canonical quasinorm associated with $A$ as follows:
    \begin{align*}
    \rho_A(x) =
            \begin{cases}
            b^j     &\textrm{ if } x \in B_{j + 1} \backslash B_j \\
            0       &\textrm{ if } x = 0.
            \end{cases}
    \end{align*}

By setting $\omega$ to be the smallest integer so that $2B_0 \subset A^{\omega} B_0 = B_{\omega}$, $\rho_A$ is a quasinorm with the doubling constant $c = b^{\omega}$. Once $A$ is fixed, we will drop the subscript and $\rho$ will always denote the step norm. The anisotropic quasi-norm is related to the Euclidean structure by the following lemma of Lemarie-Rieusset \cite{MR1286477}.

\begin{lem} Suppose $\rho_A$ is a homogeneous quasi-norm associated with dilation $A$. Then there is a constant $c_A$ such that:
    \begin{equation}
    \label{EQ1}
    \begin{aligned}
    \frac{1}{c_A} \rho_A(x)^{\zeta_-} \leq |x| \leq c_A \rho_A(x)^{\zeta_+}     \ &\textrm{ if } \ \rho_A(x) \geq 1, \\
    \frac{1}{c_A} \rho_A(x)^{\zeta_+} \leq |x| \leq c_A \rho_A(x)^{\zeta_-}     \ &\textrm{ if } \ \rho_A(x) < 1.
    \end{aligned}
    \end{equation}
where $c_A$ depends only on the eccentricities of $A$:  $\displaystyle \zeta_{\pm} = \frac{\ln \lambda_{\pm}}{\ln b}$.
\end{lem}
Lastly, we observe that if $A^{*}$ is the adjoint of $A$, then $A^{*}$ is also a dilation matrix with its own (canonical) norm $\rho_{\ast}$, though $A^*$ and $A$ have the same eigenvalues and eccentricities.

We denote $\mathcal{S}$ as the Schwartz, and $\mathcal{S}'$ the space of tempered distributions. Suppose we fix $\varphi \in \mathcal{S}$ such that $\int \varphi \ dx \neq 0$. If $k \in \Z$, we denote the anisotropic dilation by $\varphi_k (x) = b^{k} \varphi(A^{k} x)$. Then the radial maximal function on $f \in \mathcal{S}'$ is given by
    \[ M_{\varphi}^0 f(x) = \sup_{k \in \Z} |f \ast \varphi_k (x)|. \]
The anisotropic Hardy space $H_A^p$ consists of all tempered distributions $f \in \mathcal{S}'$ so that $M_{\varphi} f \in L^p$, with $\| f \|_{H_A^p} \simeq \| M_{\varphi} f \|_{L^p}$. Analogous to the isotropic setting, this definition is independent of the choice of $\varphi$ and is equivalent to the grand maximal function formulation (see \cite[Theorem 7.1]{MR1982689}).

We now present the atomic and molecular decompositions of $H_A^p$, which greatly simplifies the analysis of Hardy spaces. For a fixed dilation $A$, we say $(p, q, s)$ is an admissible triple if $p \in (0, 1]$, $1 \leq q \leq \infty$ with $p < q$, and $s \in \N$ satisfies $s \geq \left\lfloor \left( \frac{1}{p} - 1 \right) \frac{\ln b}{\ln \lambda_-} \right\rfloor$. For the rest of this article, $(p, q, s)$ will always denote an admissible triple. A $(p, q, s)$ atom is a function $a$ supported on $x_0 + B_j$ for some $x_0 \in \R^n$ and $j \in \Z$, satisfying size condition $\| a \|_q \leq |B_j|^{\frac{1}{q} - \frac{1}{p}}$, and vanishing moments condition: For $|\alpha| \leq s$, 
	\[\ds \int_{\R^n} a(x) x^{\alpha} dx = 0. \]
The following theorem is the atomic decomposition of $H_A^p$, see \cite[Theorem 6.5]{MR1982689}:
    \begin{thrm}\label{thm:atom} Suppose $p \in (0, 1]$ and $(p, q, s)$ is admissible. Then $f \in H_A^p(\R^n)$ if and only if
        \[ f = \sum_{i} \lambda_i a_i, \]
    for some sequence $( \lambda_i )_i \in \ell^p$ and $( a_i )$ a sequence of $(p, q, s)$ atoms. Moreover,
        \[ \| f \|_{H_A^p} \simeq \inf \{ \| (\lambda_i) \|_{\ell^p}: f = \sum_{i} \lambda_i a_i \}, \]
    where the infimum is taken over all possible atomic decompositions.
    \end{thrm}

We can also decompose $f \in H_A^p$ with molecules, which generalize the notion of atoms. 

    \begin{definition}\label{Def:Molecules} Let $(p, q, s)$ be admissible, and fix $d$ satisfying
        \begin{align}
        \label{Mole-d}
        d > s \frac{\ln \lambda_+}{\ln b} + 1 - \frac{1}{q},
        \end{align}
    and define $\theta = (\frac{1}{p} - \frac{1}{q})/d$. Then we say a function $M$ is a $(p, q, s, d)$ molecule centered at $x_0 \in \R^n$ if it satisfies the following size and vanishing moments conditions:
        \begin{enumerate}
        \item $\ds N(M) = \| M \|_q^{1 - \theta} \ \| \rho(x - x_0)^d M \|_q^{\theta} < \infty$,
        \item $\ds \int x^{\beta} M(x) dx = 0$ for all $|\beta| \leq s$.
        \end{enumerate}
    \end{definition}
The quantity $N(M)$ is the molecular norm of $M$. We say the quadruple $(p, q, s, d)$ is admissible if the triple $(p, q, s)$ is an admissible triple and $d$ satisfies \eqref{Mole-d}. If we say $M$ is a molecule, then it implicitly has an admissible quadruple. A straightforward computation shows that if $a$ is an atom, then $N(a) \leq C$, where $C$ is a uniform constant.

The following theorem gives the molecular decomposition of Hardy spaces. It is not new, since the crucial ideas are implicit in Lemma 9.3 of \cite{MR1982689}, though our definition of molecules is more general than what is used there. For completeness, we will include the proof in the last section.
\begin{thrm}\label{Thm:Molecular} Every molecule $M$ is in $H_A^p$, and satisfies
    \begin{align}\label{est:Molecular}
    \| M \|_{H_A^p} \leq C N(M),
    \end{align}
where $C = C(A, p, q, s, d)$. Moreover, $f \in H_A^p(\R^n)$ if and only if there exist $(p, q, s, d)$ molecules $\{ M_j \}_j$ such that $f = \sum_j M_j$ in $\mathcal{S}'$, and $\sum_j N(M_j)^p < \infty$.
In this case, we have
    \[ \| f \|_{H_A^p}^p \leq C \sum_j N(M_j)^p. \]
\end{thrm}
\section{Proof of the Multiplier Theorem \ref{Thm-Mult1}}

In proving the multiplier operator is bounded on $H_A^p$, we will follow this outline. 
	\begin{enumerate}
	\item Show that our multiplier operator is a convolution operator of a certain regularity. This is the key result of this paper, given by Lemma \ref{lem:DW}. 
	\item As is often the case with Hardy spaces, we show it suffices to verify the action of operators on atoms. As we will see in Lemma \ref{lem:UB-atom}, we only need to consider $(p, \infty, s)$ atoms.  
	\item Lastly, by Lemma \ref{lem:MB}, we show that the action of this operator on atoms will produce molecules whose (molecular) norms are uniformly bounded. By Theorem \ref{Thm:Molecular}, this completes the proof of Theorem \ref{Thm-Mult1}.
	\end{enumerate}
	In this section, we state these lemmas, and provide a proof of Theorem \ref{Thm-Mult1} (which follows immediately). The proofs of these lemmas are in the next section. 

\vskip 0.5 cm

We start by generalizing the notion of regularity to the anisotropic setting, taken from \cite{MR1982689}. 	
\begin{definition}Let $(p, q, s)$ be admissible and let $R \in \N$ satisfy
    \begin{align}
    \label{CZ-R}
    R> \max \left\{ \left( \frac{1}{p} - 1 \right) \frac{\log b}{\log \lambda_-} , s \frac{\log \lambda_+}{\log \lambda_-} \right\},
    \end{align}
and let $K \in C^R (\R^n \backs \{ 0 \})$. We say $K$ is a Calder\'{o}n-Zygmund convolution kernel of order $R$ if there exists a constant $C$ such that for all multi-indices $\alpha$ with $|\alpha| \leq R$, and all $k \in Z$, $x \in B_{k + 1} \backs B_k$,
    \begin{align}\label{Def-CZ}
    |D_A^{-k} \p_x^{\alpha} D_A^k K(x)| \leq \frac{C}{\rho(x)}.
    \end{align}
If $K$ is such a kernel, we say $K$ satisfies CZC-$R$ and its associated singular integral operator $T$ is defined by $Tf = K \ast f$.
\end{definition}

The following lemma is our key result. 
\begin{lem}\label{lem:DW} Let $N \in \N$ and $m \in L_{loc}^1 (\R^n \backs \{ 0 \})$. Suppose $m$ satisfies the Mihlin condition \eqref{Def-M} of order $N$, and define $K$ by $K = \check{m}$. Then $K$ is a Calder\'{o}n-Zygmund convolution kernel of order $R$ provided $R \in \N$ and satisfies
    \begin{align}\label{Lem-R}
    0 \leq R <  \left( N \frac{\ln \lambda_-}{\ln b} - 1 \right) \frac{\ln b}{\ln \lambda_+}.
    \end{align}
\end{lem}

The general method in proving an operator $T : H_A^p \rightarrow H_A^p$ is bounded is to show that $T$ is uniformly bounded on all $(p, q, s)$ atoms, that is, $\| T a \|_{H_A^p} \leq C$ where $a$ is a $(p, q, s)$ atom. However, as we see in \cite{MR2163588}, in general it is not sufficient to deal with $(p, \infty, s)$ atoms, though by the work of Meda et al \cite{MR2399059}, it suffices if $q < \infty$. This suggests that we simply need to show our operator satisfies $\| T a \|_{H_A^p} < \infty$ for $(p, 2, s)$ atoms.

However, this approach will not work for us, because of the following complication. Observe that we have the inclusions of the subspaces
    \[ H_{fin}^{p, 2} \subseteq L^2 \cap H_A^p \subseteq H_A^p. \]
Suppose we use the approach outlined above, and after verifying $\| Ta \|_{H_A^p} \leq C$ for all $(p, 2, s)$ atoms, we can then extend $T : H_{fin}^{p, 2} \rightarrow H_A^p$ to the unique bounded extension $\tilde{T} : H_A^p \rightarrow H_A^p$. Next, consider the operator $T_m$ on the (middle) subspace $L^2 \cap H_A^p$, which we initially defined by $T_m f = (m \hat{f})^{\vee}$. It is not clear that the extension $\tilde{T}$ will agree with $T_m$ on $(L^2 \cap H_A^p) \backslash H_{fin}^{p, 2}$. Because of this uncertainty, we cannot conclude that $\tilde{T}$ is indeed the extension of $T_m$ on $H_A^p$.

Fortunately, for multiplier operators, we have another approach, aided by a regularity result of \cite[Theorem 1]{MR3043011}. This approach also shows that it suffices, at least in our case, to verify uniform boundedness of $(p, \infty, s)$ atoms.
\begin{lem}\label{lem:UB-atom} Suppose $(p, \infty, s)$ is an admissible triple, $m \in L^{\infty}$ and $T_m$ is the associated multiplier operator initially defined on $L^2 \cap H_A^p$. Then $T_m$ has a a unique, bounded extention $\tilde{T}_m : H_A^p \rightarrow H_A^p$ if for all $(p, \infty, s)$ atoms,
	\[ \| T_m a \|_{H_A^p} \leq C, \]
where $C$ is independent of the atom $a$.
\end{lem}

This last lemma (and the regularity condition \eqref{Def-CZ}) first appeared in \cite[Theorem 9.8]{MR1982689} for the more general Calder{\'o}n-Zygmund operators. We give an alternate proof using Theorem \ref{Thm:Molecular}.

\begin{lem}\label{lem:MB}  Let $R \in \N$. Suppose $T$ is a singular integral operator whose kernel $K$ is a Calder\'{o}n-Zygmund convolution kernel of order $R$. Then $T : H_A^p (\R^n) \rightarrow H_A^p (\R^n)$ is bounded provided $p$ satisfies
    \begin{align}\label{SIO-prange}
    0 < \frac{1}{p} - 1 < R \left( \frac{(\log \lambda_-)^2}{\log b \log \lambda_+} \right).
    \end{align}
    \end{lem}

Now that all the pieces are here, we can prove Theorem \ref{Thm-Mult1}.

\begin{proof}[Proof of Theorem \ref{Thm-Mult1}] Suppose $m$ is a multiplier satisfying the Mihlin condition \eqref{Def-M} of order $N$, and $L \not\in \N$. Then by Lemma \ref{lem:DW}, we have a kernel $K$ of order $R$, satisfying \eqref{Lem-R} such that $\hat{K} = m$. Then by Lemma \ref{lem:MB}, the operator $Tf = K \ast f$ satisfies the bound $\| T a \|_{H_A^p} \leq C$ for all $(p, \infty, s)$ atoms, which by Lemma \ref{lem:UB-atom}, gives a unique extension $\tilde{T} : H_A^p \rightarrow H_A^p$, provided $p$ is in the range \eqref{SIO-prange}, which implies the range given in Theorem \ref{Thm-Mult1}.

However, if $L \in \N$, then $\lceil L - 1 \rceil \leq \lfloor L \rfloor$. To make the above argument hold, recall Remark \ref{def:tighten}, and let $\tilde{\lambda}_-$ and $\tilde{\lambda}_+$ be defined so that 
	\[ 1 < \lambda_- < \tilde{\lambda}_- < |\lambda_1| \leq \ldots \leq |\lambda_n| < \tilde{\lambda}_+ < \lambda_+, \]  
	so that the new $\tilde{L}$, defined in terms of the new eccentricities, is slightly larger, and no longer an integer. However, $\lfloor \tilde{L} \rfloor = \lfloor L \rfloor$, and we can repeat the above argument and obtain the bound \eqref{range:multiplier}. 
\end{proof}

\section{Proofs of Lemmas and the Molecular Decomposition} 

In this section, we give the proofs of Lemma \ref{lem:DW}, \ref{lem:UB-atom}, and \ref{lem:MB}, as well as the proof of Theorem \ref{Thm:Molecular}. Lemma \ref{lem:DW} is the key result of this paper. Lemma \ref{lem:MB} and Theorem \ref{Thm:Molecular} originally appear in \cite{MR1982689}, and we reprove it here with our notion of molecules.

\subsection{Proof of lemmas}

\begin{proof}[Proof of Lemma \ref{lem:DW}] Let $m$ satisfy the Mihlin condition of order $N$ and let $R$ satisfy \eqref{Lem-R}. Fix $\Psi \in S(\R^n)$ such that $\hat{\Psi}$ is supported on $B_1^{\ast} \backslash B_{-1}^{\ast}$, and for all $\xi \neq 0$,
    \begin{align*}
    \sum_{j \in \Z} \hat{\Psi}(A^{-j} \xi) = 1.
    \end{align*}
By setting $\Psi_j (x) = b^j \Psi (A^j x)$, we have the identity $\widehat{\Psi_j}(\xi) = D_{A^{\ast}}^{-j} \hat{\Psi} (\xi) = \hat{\Psi}((A^{*})^{ -j}) \xi)$, and $\widehat{\Psi_j}$ is supported on $B_{j + 1}^{\ast} \backslash B_{j - 1}^{\ast}$. We define $m_j = m \widehat{\Psi}_j$, which is supported on $B_{j + 1}^{\ast} \backslash B_{j -1}^{\ast}$, and define $K_j = (m_j)^{\vee}$. Then we have
        \begin{align*}
        m 	&= \sum_{j \in \Z} m_j \qquad \text{ holds pointwise and in } \mathcal{S}', \text{ and } \qquad
        K 	= \sum_{j \in \Z} K_j \text{ in } \mathcal{S}'.
        \end{align*}
We will see that the equality for $K$ also holds pointwise. We make the following reductions to prove the CZC-$R$ condition \eqref{Def-CZ}. First, it suffices to show that for all multi-index $\beta$ such that $|\beta| \leq R$, $k \in \Z$, and $x \in B_{1} \backs B_0$, $|\p_x^{\alpha} D_A^k K(x)| \leq C/b^{k}$, which follows from the absolute convergence
    \[
    \sum_{j \in \Z} |\p_x^{\beta} D_A^k K_j (x)| \leq \frac{C}{b^{k}}.
    \]
To prove this, it suffices to prove the above convergence for $k = 0$:
    \begin{align}\label{CZ-Goal}
    \sum_{j \in \Z} |\p_x^{\beta} K_j (x)| \leq C.
    \end{align}
Indeed, suppose \eqref{CZ-Goal} holds. Then if $k \in \Z$, and $m$ has the Mihlin property, then so does $D_{A^{\ast}}^k m$, with the same constant $C$. Therefore if $\xi \in B_{j + 1}^{\ast} \backslash B_j^{\ast}$, then $(A^*)^ k \xi \in B_{j + k + 1}^* \backslash B_{j + k}^*$, so
        \begin{align*}
        |(D_{A^{\ast}}^{-j} \, \p_{\xi}^{\beta} \, D_{A^{\ast}}^j) (D_{A^{\ast}}^k m)(\xi)|
            &= |(D_{A^{\ast}}^{-j - k} \, \p_{\xi}^{\beta} \, D_{A^{\ast}}^{j + k} m((A^*)^k  \xi)| \leq C_{\beta}.
        \end{align*}
To prove \eqref{CZ-Goal}, we decompose the sum using a well-chosen integer $M$. Denote $\lambda_{\max}^{\ast}$ as the eigenvalue of $A^{\ast}$ with the largest norm and $\| \cdot \|_{\mathrm{op}}$ is the operator norm on $\R^{n} \rightarrow \R^{n}$. By the spectral theorem,
    \[ \lambda_{\max}^{\ast} = \limsup_{j \rightarrow \infty} \| A^{\ast j} \|_{\mathrm{op}}^{1/j}. \]
Let $\epsilon > 0$. Then there exists an integer $M > 0$ such that for all $j > M$,
    \[ \| A^{\ast j} \|_{\mathrm{op}}^{1/j} \leq (1 + \epsilon) \lambda_{\max}^{\ast} \leq (1 + \epsilon) \lambda_+ . \]
With this $M$, we write
    \[ \sum_{j \in \Z} |\p_x^{\beta} K_j (x)| = \sum_{j \leq M} |\p_x^{\beta} K_j (x)| + \sum_{j > M} |\p_x^{\beta} K_j (x)| = S_L + S_H. \]
We call $S_L$ and $S_H$ the low and high spatial terms, respectively. Starting with the high spatial terms, we fix $j > M$ and $x \in B_{-1} \backs B_0$. Then we can fix another multi-index $\alpha$ satisfying $|\alpha| = N$ such that there exists a constant $c$ depending only on $n$ such that $|(A^j x)^{\alpha}| \geq c|A^j x|^N$. This can be done by picking $\alpha = Ne_i$ where $e_i$ is the $i^{th}$ unit vector in the canonical basis of $\R^n$ and the direction $i$ is where $A^j x$ has the largest value in norm. Define  $w(u) = (A^{\ast j} u)^{\beta} m(A^{\ast j} u) \hat{\Psi}(u)$. Using Parseval's identity, integration by parts, and a change of variables, we have
    \begin{align*}
    \p_x^{\beta} K_j (x) = c_{\beta} b^j \int_{B_1^{\ast} \backs B_{-1}^{\ast}} (\p_u^{\alpha} w)(u) \frac{e^{2\pi i \langle A^j x, u \rangle}}{(2\pi i A^j x)^{\alpha}} du,
    \end{align*}
which we estimate using the bound from the spectral theorem.

Then the product rule gives:
    \begin{align}
    \label{I1I2}
    (\p^{\alpha} w)(u) = \sum_{\gamma \leq \alpha} {\alpha \choose \gamma} \underbrace{\p^{\alpha - \gamma} (D_{A^{\ast}}^j m \cdot \hat{\Psi}) (u)}_{I_1} \cdot \underbrace{\p^{\gamma} ((A^{\ast j} u)^{\gamma})}_{I_2}.
    \end{align}
By another application of the product rule, we have a uniform constant $c'$, independent of $m$, $j$, $u$ such that
    \begin{align*}
    I_1 \leq c \, \sup_{\delta \leq \gamma} |\p^{\delta} D_{A^{\ast}}^j m)(u)|
    &= c \, \sup_{\delta \leq \gamma} |(D_{A^{\ast}}^{-j} \p^{\delta} D_{A^{\ast}}^j m)(A^{\ast j} u)| \leq c'.
    \end{align*}
We now bound $I_2$. With $u \in B_1^{\ast} \backs B_{-1}^{\ast}$, elementary considerations from expressing $(A^{\ast j} u)^{\beta}$ as a sum of monomials show that there exists $c$ depending only on $N$, such that by our choice of $M$ and $j > M$,
    \begin{align*}
    I_2 = |\p^{\gamma}(A^{\ast j} u)^{\beta}| \leq c \| A^{\ast j} \|_{\mathrm{op}}^{|\beta|} \leq c (\lambda_{+}^{\ast}(1 + \epsilon))^{j |\beta|},
    \end{align*}
Combining our estimates of $I_1$ and $I_2$ in \eqref{I1I2}, we have a constant $C$, depending on the past constants, such that
    \[ |(\p^{\alpha} w)(u)| \leq C (\lambda_+^{\ast} (1 + \epsilon))^{j |\beta|}. \]
Then we have
    \begin{align*}
    |\p^{\beta} K_j (x)|
            &\leq b^j \int_{B_1^{\ast} \backslash B_{-1}^{\ast}} \left| \frac{(\p^{\alpha} w) (u)}{(2\pi i A^j x)^{\alpha}} \right| du \leq C \left( \frac{b^j (\lambda_+^{\ast}(1 + \epsilon))^{j|\beta|}}{|A^j x|^{|\alpha|}} \right) \leq C \left( \frac{b^j (\lambda_+^{\ast}(1 + \epsilon))^{j |\beta|}}{b^{j |\alpha| \zeta_-}} \right).
    \end{align*}
Note that with our choice of $\alpha$ and $\eqref{EQ1}$, we can sum $|\p^{\beta} K_j (x)|$ for $j > M$ if
    \begin{align*}
    \frac{b(\lambda_+^{\ast}(1 + \epsilon))^{|\beta|}}{b^{|\alpha| \zeta_-}} < 1, \qquad \textrm{ that is, } \qquad |\beta| < \left( \frac{N \log \lambda_-}{\log b} - 1 \right) \frac{\log b}{\log(\lambda_+^{\ast}(1 + \epsilon))}.
    \end{align*}
Indeed, for $|\beta| \leq R$, there exists $\epsilon > 0$ such that the series below converges: For $C_1$ depending only on $A, n, \Psi, \beta, M$, we have
    \[ \sum_{j = M + 1}^{\infty} |\p^{\beta} K_j (x)| \leq C \sum_{j = M + 1}^{\infty} \left( \frac{b(\lambda_+^{\ast}(1 + \epsilon))^{|\beta|}}{b^{|\alpha| \zeta_-}} \right)^j \leq C_1.
    \]
Turning our attention to $S_L$, we start with Parseval's identity and a change of variables. With $C$ a dimensional constant, we have
    \begin{align*}
    |\p_x^{\beta} K_j (x)|
        &= \left| \int_{B_{j + 1}^{\ast} \backs B_{j - 1}^{\ast}} (2\pi i \xi)^{\beta} m_j (\xi) e^{2\pi i \langle x, \xi \rangle} d\xi \right| \\
        &\leq cb^j \int_{B_1^{\ast} \backs B_{-1}^{\ast}} \underbrace{|(A^{\ast j} u)^{\beta}|}_{J_1} \cdot \underbrace{|m_j (A^{\ast j} u)|}_{J_2} du
        \leq
        \begin{cases}
        C b^{j(1 + |\beta| \zeta_+)} &\textrm{ if } j \geq 0 \\
        C b^{j(1 + |\beta| \zeta_-)} &\textrm{ if } j < 0.
        \end{cases}
    \end{align*}
Indeed, for $u$ in the unit annulus $B_1^* \backs B_{-1}^*$, we have $A^{\ast j} u \in B_{j + 1}^{\ast} \backs B_{j -1}^{\ast}$, $J_1 \leq c |A^{\ast j} u|^{|\beta|} \leq C b^{j \zeta_{\pm}} \rho_{\ast} (u)^{\zeta_{\pm}}$, with the eccentricity $\zeta_{\pm}$ depending on the sign of $j$. Since $m \in L^{\infty}$ and $J_2 \leq C(m, \Psi)$, we obtain the above estimate. Returning to $S_L$, we have a constant $C$, depending only on $n, A, N, \Psi, M$ such that
    \begin{align*}
    S_L \leq \sum_{j = -\infty}^{-1} |\p_x^{\beta} K_j (x)| + \sum_{j = 0}^M \ |\p_x^{\beta} K_j (x)| \leq C \sum_{j = -\infty}^{-1} b^{j(1 + |\beta| \zeta_-)} + C \sum_{j = 0}^M b^{j(1 + |\beta| \zeta_+)} \leq C_2,
    \end{align*}
with $C_2 = C_2 (n, A, N, \Psi, M)$. This completes the estimate \eqref{CZ-Goal}, and this proof.
\end{proof}


For the proof of Lemma \ref{lem:UB-atom}, we need the following result from \cite{MR3043011}.

\begin{thrm}(\cite[Theorem 1]{MR3043011}) \label{lem:MBDW1} Let $p \in (0, 1]$. If $f \in H_A^p$, then $\hat{f}$ is a continuous function and satisfies
	\[ |\hat{f}(\xi)| \leq C \| f \|_{H_A^p} \rho_{\ast} (\xi)^{\frac{1}{p} - 1}. \]
In particular, if $f = \sum_j \lambda_j a_j$, then $\hat{f}(\xi) = \sum_j \lambda_j \hat{a_j} (\xi)$ almost everywhere and in $\mathcal{S}'$.
\end{thrm}

\begin{proof}[Proof of Lemma \ref{lem:UB-atom}] If $a$ is a $(p, \infty, s)$ atom, it is compactly supported, so that it is in $L^2 \cap H_A^p$, and $T_m a = (m \hat{a})^{\vee}$ is well-defined. Now let $f \in L^2 \cap H_A^p$, with an infinite atomic decomposition $f = \sum_j \lambda_j a_j$ using $(p, \infty, s)$ atoms. We first establish $T_m$ can pass through the infinite sum:
	\begin{align}\label{T:Passes}
	T_m f = T_m \big( \sum_j \lambda_j a_j \big) = \sum_j \lambda_j T_m a_j.
	\end{align}
Observe that passing the operator through the infinite sum is the main issue raised by \cite{MR2163588} and the rationale as to why the result of \cite{MR2399059} is needed for a general sublinear operator. In our case where our operator is a multiplier, we show that we can do this directly. If we denote the right-hand term above by $g = \sum_j \lambda_j T_m a_j$, then \eqref{T:Passes} holds if we can show $(T_m f)^{\wedge} = m\hat{f} = \hat{g}$ in $\mathcal{S}'$. To show \eqref{T:Passes}, we note that by Lemma \ref{lem:MBDW1}, for $\xi \in \R^n$ almost everywhere, we have $(m\hat{f}) (\xi) = \sum_j \lambda_j m(\xi) \hat{a_j} (\xi)$. Then
	\begin{align*}
	\hat{g}
		&= \big( \sum_j \lambda_j T_m a_j \big)^{\wedge} = \sum_j \lambda_j (T_m a_j)^{\wedge} = \sum_j \lambda_j (m \hat{a_j}) = m \sum_j \lambda_j \hat{a_j} = m \hat{f}.
	\end{align*}
Since $\hat{g} = m\hat{f} = (T_m f)^{\wedge}$ in $\mathcal{S}'$ and pointwise, we must also have $g = T_m f$, thus establishing the equality \eqref{T:Passes}. The boundedness of $T_m : L^2 \cap H_A^p \rightarrow H_A^p$ follows immediately:
	\[ \left\| T_m f \right\|_{H_A^p}^p = \| \sum_j \lambda_j T(a_j) \|_{H_A^p}^p \leq \sum_j |\lambda_j|^p \| T(a_j) \|_{H_A^p}^p \leq C \sum_j |\lambda_j|^p. \]
Taking the infimum over all possible atomic decompositions, we have $\| T_m f \|_{H_A^p} \leq C \| f \|_{H_A^p}$. Lastly, with $L^2 \cap H_A^p$ a dense subset of $H_A^p$, there exists a unique bounded extension $\tilde{T}_m : H_A^p \rightarrow H_A^p$ such that $\tilde{T}_m = T_m$ on $L^2 \cap H_A^p$.
\end{proof}

    \begin{proof}[Proof of Lemma \ref{lem:MB}]
    Let $p$ satisfy \eqref{SIO-prange} and $(p, q, s)$ is an admissible triple. Let $a$ be a $(p, q, s)$ atom supported on the ellipoid $x_0 + B_r$ for some $x_0 \in \R^n$ and $r \in \Z$. The boundedness of $T$ follows once we establish the uniform bound of the molecular norm $N(Ta) = \| Ta \|_q^{1 - \theta} \| \rho(x - x_0)^d Ta(x) \|_q^{\theta} \leq C$. Note that since $(\R^n, dx, \rho)$ is a space of homogeneous type, $T$ is bounded from $L^q$ to $L^q$ for $q > 1$. There is a $C$, depending only on $T$, $q$, and $\theta$, such that
        \[ \| T a \|_q^{1 - \theta} \leq C \| a \|_q^{1 - \theta} \leq C b^{r(\frac{1}{q} - \frac{1}{p})(1 - \theta)}. \]
By Minkowski's inequality:
         \begin{align*}
         \| \rho(x - x_0)^d Ta(x) \|_q  &\leq \big( \int_{x_0 + B_{r + 2\omega}} |\rho(x - x_0)^{dq} Ta(x)|^q dx \big)^{1/q} \\
                                        &+ \big( \int_{(x_0 + B_{r + 2\omega})^c} |\rho(x - x_0)^{dq} Ta(x)|^q dx \big)^{1/q} = I_1 + I_2
         \end{align*}

The estimate for $I_1$ is immediate:
        \begin{align*}
        I_1 \leq b^{d(r + 2\omega)} \left( \int_{x_0 + B_{r + 2\omega}} |Ta(x)|^q dx \right)^{1/q} \leq b^{d(r + 2\omega)} \| Ta \|_q \leq C b^{dr} b^{r(\frac{1}{q} - \frac{1}{p})} = C b^{r(d + \frac{1}{q} - \frac{1}{p})}.
        \end{align*}
To estimate $I_2$, we require the following pointwise estimate from \cite[Lemma 9.5]{MR1982689}: Suppose $T$ is a singular integral operator whose kernel $k$ is CZC-$R$, with $R$ satisfying \eqref{CZ-R}. Then there exists a constant $C$ such that for every $(p, q, s)$ atom $a$ with support $x_0 + B_r$, all $l \geq 0$ and $x \in x_0 + (B_{r + l + 2\omega + 1} \backslash B_{r + l + 2\omega})$,
    \begin{align*}
    |Ta(x)| \leq C b^{-l R \zeta_-  - l} |B_r|^{-1/p}.
    \end{align*}
    With this estimate, we have
        \begin{align*}
        I_2 &= \sum_{j = 0}^{\infty} \int_{x_o + (B_{r + 2\omega + j + 1} \backslash B_{r + 2\omega + j})} \rho(x - x_o)^{dq} |Ta(x)|^q dx \\
            &\leq C b^{-\frac{rq}{p}} \sum_{j = 0}^{\infty} b^{-jq (1 + R \zeta_-)} \left( \int_{x_o + B_{r + 2\omega + j + 1} \backslash B_{r + 2\omega + j}} \rho(x - x_o)^{dq} dx \right) \\
            &= C b^{-\frac{rq}{p}} \sum_{j = 0}^{\infty} b^{-jq (1 + R \zeta_-)} b^{(dq + 1)(r + 2\omega + j)} = C b^{r(dq + 1 - \frac{q}{p})} \sum_{j = 0}^{\infty} b^{j(dq + 1 - q(1 + R \zeta_-))}.
        \end{align*}
    The geometric series converges exactly when $R$ satisfies \eqref{SIO-prange}. Taking the power $\theta/q$ on both sides, we have
        \[ \| \rho(x - x_o)^d Ta(x) \|_q^{\theta} \leq C b^{r\theta(d + \frac{1}{q} - \frac{1}{p})}. \]
    All together, we have $N(Ta) \leq C b^{k(\frac{1}{q} - \frac{1}{p})(1 - \theta)} b^{k\theta (d + \frac{1}{q} - \frac{1}{p})} = C$, as the exponent is exactly 0.
    \end{proof}

\subsection{Proof of Theorem \ref{Thm:Molecular}} 

We need a few preliminary results on projections and molecules, which we state without proof as they are implicit in the proof of Lemma 9.3 of \cite{MR1982689}. To define the projections needed, recall that given a dilation $A$, $\{ B_j \}_{j \in \Z}$ denotes the `canonical' ellipsoids so that for all $j \in \Z$, $A (B_j) = B_{j + 1}$. We also define $\| f \|_{L^1 (B)} = \int_B |f(x)| dx$.

    \begin{definition} Let $s \in \N$ and $\mathcal{B} = \{ x + B_j : x \in \R^n, j \in \Z \}$. Define $P_s$ to be the space of polynomials on $\R^n$ of degree at most $s$.  If $B \in \mathcal{B}$, we define $\pi_B$ as the natural projection defined by the Riesz Lemma:
            \[ \int_{B} (\pi_B f(x)) Q(x) dx = \int_{B} f(x) Q(x) dx, \qquad \textrm{ for all } f \in L^1(B) \textrm{ and } Q \in P_s. \]
\end{definition}

With these projections, we make some elementary observations. Let $Q = \{ Q_{\alpha} \}_{|\alpha| \leq s}$ be an orthonormal basis of $P_s$ in $L^2(B_0)$-norm, that is, $\langle Q_{\alpha}, Q_{\beta} \rangle = \int_{B_0} Q_{\alpha} (x) \overline{Q_{\beta} (x)} dx = \delta_{\alpha, \beta}$. Then the projection $\pi_{B_0} : L^1(B_0) \rightarrow P_s$ is given by
         \begin{align*}
         \pi_{B_0} f = \sum_{|\alpha| \leq s} \left( \int_{B_0} f(x) \overline{Q_{\alpha} (x)} dx \right) Q_{\alpha}.
         \end{align*}
     Generally, if $j \in \Z$, then $\pi_{B_j} : L^1 (B_j) \rightarrow P_s$ is given by
        \begin{align}
        \label{Proj-j}
        \pi_{B_j} f = \left( D_A^{-j} \pi_{B_0} D_A^j \right) f.
        \end{align}
    If $B = y + B_j$, then $\pi_{B} : L^1 (B) \rightarrow P_s$ is given by
        $\pi_B f = \left( T_y \pi_{B_j} T_{-y} \right) f$,
        where $T_y f = f(x - y)$ is the translation operator, and there exists $C_0$, depending only on $s$ and $Q$, such that for all $B \in \mathcal{B}$, given $x \in B$,
        \begin{align}\label{Proj-UnifBound}
        |\pi_B f (x)| \leq C_0 \int_{B} |f| \ \frac{dx}{|B|}.
        \end{align}
    Let $\tilde{\pi}_B = \mathrm{Id} - \pi_B$ be the complementary projection. Then for all $B \in \mathcal{B}$, $\tilde{\pi}_B : L^q (B) \rightarrow L^q (B)$ is bounded, with
        \begin{align*}
        \| \tilde{\pi}_B (f) \|_{L^q (B)} \leq (1 + C_0) \| f \|_{L^q (B)}.
        \end{align*}
    Furthermore, for all $\alpha$ with $|\alpha| \leq s$, we have $ \int_{B} x^{\alpha} \cdot (\tilde{\pi}_B f) (x) dx = 0$.

\begin{lem}\label{lem:projM} Let $M$ be a $(p, q, d)$ molecule centered at $x_0$.
        \begin{enumerate}
        \item Then $\| \pi_j M \|_{L^1(B_j)} \rightarrow 0$ as $j \rightarrow \infty$.
        \item Define $g_j = (\tilde{\pi}_{B_j} M) \textbf{1}_{B_j} = (M - \pi_{B_j} M) \textbf{1}_{B_j}$. Then $g_j \rightarrow M$ in $L^1$ as $j \rightarrow \infty$.
        \end{enumerate}
\end{lem}

\begin{proof}[Proof of Theorem \ref{Thm:Molecular}] We first prove estimate \eqref{est:Molecular}. Let $M$ be a $(p, q, d)$ molecule. Without loss of generality, we assume $N(M) = \| M \|_{q}^{1 - \theta} \| M(u) \rho(u)^d \|_{q}^{\theta} = 1$. Define the quantity $\sigma$ by $\| M \|_q = \sigma^{\frac{1}{q} - \frac{1}{p}}$ and and choose $k \in \Z$ such that $b^{k} \leq \sigma < b^{k + 1}$. From lemma \ref{lem:projM}, we have the following expression for $M$, with convergence in $L^1$:
        \[ M = g_k + \sum_{j = k}^{\infty} (g_{j + 1} - g_j). \]
    Note that for each $j$, $g_j$ has vanishing moments of order up to $s$, and has compact support. We will decompose $M$ by setting $g_k = \mu_k a_k$ and $g_{j + 1} - g_j = \mu_j a_j$, where $(\mu_j)_{j = k}^{\infty} \in \ell^p$ has a uniform norm independent of $M$ and $(a_j)_{j = k}^{\infty}$ is a sequence of  $(p, q, s)$ atoms. We start with $g_k = (M - \pi_k M) \textbf{1}_{B_k}$. With $C_0$ as in \ref{Proj-UnifBound}, we have
        \[ \| g_k \|_{L^q(B_k)} \leq \| M \|_{L^q (B_k)} + \| \pi_k M \|_{L^q(B_k)} \leq (1 + C_0) \| M\|_{L^q(B_k)}. \]
    Scaling the measure, we obtain
        \[ \| g_k \|_{L^q \left( \frac{\chi_{B_k}}{|B_k|} dx \right)} = \left( \int_{B_k} |g_k (x)|^q\frac{dx}{|B_k|} \right)^{1/q} \leq (1 + C_0) \| M\|_{L^q(\frac{dx}{|B_k|})}. \]
     Note that because $\frac{1}{q} - \frac{1}{p} < 0$, we have $\sigma \geq b^k \Rightarrow \sigma^{\frac{1}{q} - \frac{1}{p}} \leq b^{k(\frac{1}{q} - \frac{1}{p})}$. Continuing our estimate using the definition of $\sigma$, we have
        \begin{align*}
        \| M \|_{L^q(\frac{dx}{|B_k|})}     &= |B_k|^{-1/q} \| M \|_q = |B_k|^{-1/q} \sigma^{\frac{1}{q} - \frac{1}{p}}
                                            \leq |B_k|^{-\frac{1}{q}} |B_k|^{\frac{1}{q} - \frac{1}{p}} = |B_k|^{\frac{1}{p}}.
        \end{align*}
Therefore we have $\| g_k \|_{q} \leq (1 + C_0) |B_k|^{\frac{1}{q} - \frac{1}{p}}$, which gives $g_k = \mu_k a_k$ where $a_k$ is a $(p, q, s)$ atom and $\mu_k = 1 + C_0$. For $j > k$, we have
        \[ g_{j + 1} - g_j = \textbf{1}_{B_{j + 1} \backslash B_j} M - \pi_{B_{j + 1}} M + \pi_{B_j} M. \]
Estimating the first term, we have
        \begin{align*}
        \| M \textbf{1}_{B_{j + 1} \backslash B_j} \|_{L^q(\frac{dx}{|B_{j + 1}|})}
            = |B_{j + 1}|^{-\frac{1}{q}} b^{-jd} \left( \int_{B_{j + 1} \backslash B_j} |M(x)|^q \rho(x)^{dq} dx \right)^{1/q} \\
            \leq b^d |B_{j + 1}|^{-\frac{1}{q}} |B_{j + 1}|^{-d} \| M(x) \rho(x)^d \|_q = b^d |B_{j + 1}|^{-\frac{1}{q} - d} \| M \|_q^{\frac{\theta - 1}{\theta}} \leq b^d |B_{j + 1}|^{-\frac{1}{p}} b^{(k - j) d(1 - \theta)}.
        \end{align*}
    Setting $r = d(1 - \theta) > 0$, we obtain the estimate $\| M \textbf{1}_{B_{j + 1} \backslash B_j} \|_q \leq b^d |B_{j + 1}|^{\frac{1}{q} - \frac{1}{p}} b^{(j - k)(-r)}$. Next, we estimate $\pi_{B_j} M$ with Minkowski's inequality:
        \begin{align*}
        \| \pi_{B_j} M \|_{L^q (B_j)}
            &= \bigg\| \sum_{|\alpha| \leq s} \left( \int_{B_j} M(u) \overline{Q_{\alpha} (A^{-j} u)} \frac{du}{b^j} \right) Q_{\alpha} (A^{-j} x) \bigg\|_{L^q (B_j)} \\
            &\leq \sum_{|\alpha| \leq s} b^{-j} \left| \int_{B_j} M(u) \overline{Q_{\alpha} (A^{-j}u)} du \right| \| D_A^{-j} Q_{\alpha} \|_{L^q (B_j)}.
        \end{align*}
    Let $C(Q)$ a uniform bound for $\| Q_{\alpha} \|_{L^q (B_0)}$. By a change of variables, we have $\| D_{A}^{-j} Q_{\alpha} \|_{L^q(B_j)} = b^{\frac{j}{q}} \| Q_{\alpha} \|_{L^q (B_0)} \leq C(Q) b^{\frac{j}{q}}$. Next, since $M$ has vanishing moments, and $\frac{1}{q} + \frac{1}{q'} = 1$,
        \begin{align*}
        \left| \int_{B_j} M(u) \overline{Q_{\alpha} (A^{-j}u)} du \right|
            &= \left| \int_{B_j^c} M(u) \overline{Q_{\alpha} (A^{-j}u)} du \right| \leq \int_{B_j^c} |M(u)| |Q_{\alpha} (A^{-j} u)| du \\
            &\leq C(Q) \int_{B_j^c} |M(u)| |A^{-j} u|^s du \leq C(Q) c_A \int_{B_j^c} |M(u)| \rho(A^{-j} u)^{s \zeta_+} du \\
            &\leq C(Q) c_A b^{-j s \zeta_+} \left( \int_{B_j^c} |M(u)|^q \rho(u)^{dq} du \right)^{1/q} \left( \int_{B_j^c} \rho(u)^{q' (s \zeta_+ - d)} du \right)^{1/q'}.
        \end{align*}
The first integral in the last expression can be computed as follows:
         \begin{align*}
         \left( \int_{B_j^c} |M(u)|^q \rho(u)^{dq} du \right)^{1/q}
            &\leq \| M(x) \rho(x)^d \|_q = \| M \|_q^{\frac{\theta - 1}{\theta}}
            = \sigma^{\left( \frac{1}{q} - \frac{1}{p} \right)\left(\frac{\theta - 1}{\theta}\right)} = \sigma^{d(1 - \theta)} \leq  b^{kr}.
         \end{align*}
    The second integral from Holder's inequality can be computated directly as a geometric series. With   with $C$ a constant depending only on $A, q, s$, and $d$, and $d > s\zeta_+ 1 - \frac{1}{q}$, we have
        \begin{align*}
        \int_{B_j^c} \rho(u)^{q' (s \zeta_+ - d)} du
            &= \sum_{m = j}^{\infty} \int_{B_{m + 1} \backslash B_m} \rho(u)^{q' (s\zeta_+ - d)} du
            = C b^{j(1 + q'(s\zeta_+ - d))}.
        \end{align*}
    This gives
        \[ \left| \int_{B_j} M(u) \overline{Q_{\alpha} (A^{-j} u)} du \right| \leq C b^{-js\zeta_+} b^{kr} b^{j(1 - \frac{1}{q} + s\zeta_+ - d)} = C b^{kr} b^{j(1 - \frac{1}{q} - d)}. \]
    Then we have the following estimate on $\| \pi_{B_j} M \|_{L^q}$,
        \begin{align*}
        \| \pi_{B_j} M \|_{L^q}     &\leq C b^{-j} b^{kr} b^{j(1 - \frac{1}{q} - d)} b^{\frac{j}{q}} = C b^{-jd} b^{-k (-r)} = C b^{j(\frac{1}{q} - \frac{1}{p})} b^{(j - k)(-r)}.
        \end{align*}
    Finally, returning to the estimate on $g_{j + 1} - g_j$, we have
        \begin{align*}
        \| g_{j + 1} - g_j \|_q &\leq \| M \textbf{1}_{B_{j + 1} \backslash B_j} \|_q + \| \pi_{B_{j + 1}} M \|_q + \| \pi_{B_j} M \|_q \leq C |B_{j + 1}|^{\frac{1}{q} - \frac{1}{p}} b^{(j - k)(-r)}.
        \end{align*}
    Therefore if $j > k$, $g_{j + 1} - g_j = \mu_j a_j$, with $\mu_j = C b^{(j - k)(-a)}$ and where $a_j$ is a $(p, q, s)$ atom supported on $B_{j + 1}$. Summing the coefficients, we have
        \[ \sum_{j = k}^{\infty} |\mu_j|^p = \mu_k + \sum_{j = 1}^{\infty} C^p b^{-jrp} = (1 + C_0) + \frac{C}{1 - b^{-rp}}. \]
    This establishes \eqref{est:Molecular} with $C$ depending only on $A, p, q, s, d$ and the cube $Q$, and is independent of $M$.

    Lastly, we prove the molecular decomposition. If $f \in H_A^p$, then its atomic decomposition $\sum_j \lambda_j a_j$ can be seen as a molecular decomposition with $M_j = \lambda_j a_j$. Then by \eqref{est:Molecular}, we have
    \begin{align*}
    \| f \|_{H_A^p}^p &\leq \sum_j \| \lambda_j a_j \|_{H_A^p}^p \leq C \sum_j N(\lambda_j a_j)^p \leq C' \sum_j \lambda_j^p < \infty,
    \end{align*}
    where in the penultimate inequality, we used the fact that the molecular norm of atoms are uniformly bounded.

    As for the converse, suppose $f \in S'$ has the molecular decomposition $f = \sum_j M_j$ with $\sum_j N(M_j)^p < \infty$. Then again by \eqref{est:Molecular}, we have
        \[ \| f \|_{H_A^p} = \| \sum_j M_j \|_{H_A^p}^p \leq \sum_j \| M_j \|_{H_A^p}^p \leq C \sum_j N(M_j)^p < \infty. \]
    So $f \in H_A^p$, and this completes our proof.

    \end{proof}

\renewcommand{\bibname}{\textsc{references}} 
\bibliographystyle{amsplain} 			\bibliography{Cite}

\providecommand{\bysame}{\leavevmode\hbox to3em{\hrulefill}\thinspace}
\providecommand{\MR}{\relax\ifhmode\unskip\space\fi MR }
\providecommand{\MRhref}[2]{%
  \href{http://www.ams.org/mathscinet-getitem?mr=#1}{#2}
}
\providecommand{\href}[2]{#2}
\begin{thebibliography}{10}

\bibitem{ABR-2016}
V.~Almeida, J.J. Betancor, and L.~Rodriguez, \emph{Anisotropic hardy-lorentz
  spaces with variable exponents}, Preprint (2016).

\bibitem{MR776176}
A.~Baernstein, II and E.~Sawyer, \emph{Embedding and multiplier theorems for
  {$H^P({\bf R}^n)$}}, Mem. Amer. Math. Soc. \textbf{53} (1985), no.~318,
  iv+82. \MR{776176 (86g:42036)}

\bibitem{MR2720206}
{\'A}.~B{\'e}nyi and M.~Bownik, \emph{Anisotropic classes of homogeneous
  pseudodifferential symbols}, Studia Math. \textbf{200} (2010), no.~1, 41--66.
  \MR{2720206 (2012c:47131)}

\bibitem{MR1982689}
M.~Bownik, \emph{Anisotropic {H}ardy spaces and wavelets}, Mem. Amer. Math.
  Soc. \textbf{164} (2003), no.~781, vi+122. \MR{1982689 (2004e:42023)}

\bibitem{MR2163588}
\bysame, \emph{Boundedness of operators on {H}ardy spaces via atomic
  decompositions}, Proc. Amer. Math. Soc. \textbf{133} (2005), no.~12,
  3535--3542 (electronic). \MR{2163588 (2006d:42028)}

\bibitem{MR2492226}
M.~Bownik, B.~Li, D.~Yang, and Y.~Zhou, \emph{Weighted anisotropic {H}ardy
  spaces and their applications in boundedness of sublinear operators}, Indiana
  Univ. Math. J. \textbf{57} (2008), no.~7, 3065--3100. \MR{2492226
  (2010a:42078)}

\bibitem{MR2179611}
Marcin Bownik, \emph{Atomic and molecular decompositions of anisotropic {B}esov
  spaces}, Math. Z. \textbf{250} (2005), no.~3, 539--571. \MR{2179611}

\bibitem{MR2186983}
Marcin Bownik and Kwok-Pun Ho, \emph{Atomic and molecular decompositions of
  anisotropic {T}riebel-{L}izorkin spaces}, Trans. Amer. Math. Soc.
  \textbf{358} (2006), no.~4, 1469--1510. \MR{2186983}

\bibitem{MR3043011}
Marcin Bownik and Li-An~Daniel Wang, \emph{Fourier transform of anisotropic
  {H}ardy spaces}, Proc. Amer. Math. Soc. \textbf{141} (2013), no.~7,
  2299--2308. \MR{3043011}

\bibitem{MR0417687}
A.~Calder{\'o}n and A.~Torchinsky, \emph{Parabolic maximal functions associated
  with a distribution}, Advances in Math. \textbf{16} (1975), 1--64.
  \MR{0417687 (54 \#5736)}

\bibitem{MR3233216}
David Cruz-Uribe and Li-An~Daniel Wang, \emph{Variable {H}ardy spaces}, Indiana
  Univ. Math. J. \textbf{63} (2014), no.~2, 447--493. \MR{3233216}

\bibitem{MR2838119}
Shai Dekel, Pencho Petrushev, and Tal Weissblat, \emph{Hardy spaces on {$\Bbb
  R^n$} with pointwise variable anisotropy}, J. Fourier Anal. Appl. \textbf{17}
  (2011), no.~5, 1066--1107. \MR{2838119}

\bibitem{MR2316761}
Y.~Ding and S.~Lan, \emph{Some multiplier theorems for anisotropic {H}ardy
  spaces}, Anal. Theory Appl. \textbf{22} (2006), no.~4, 339--352. \MR{2316761
  (2008b:42020)}

\bibitem{MR0447953}
C.~Fefferman and E.~Stein, \emph{{$H^{p}$} spaces of several variables}, Acta
  Math. \textbf{129} (1972), no.~3-4, 137--193. \MR{0447953 (56 \#6263)}

\bibitem{MR1286477}
P.~Lemari{\'e}-Rieusset, \emph{Projecteurs invariants, matrices de dilatation,
  ondelettes et analyses multi-r\'esolutions}, Rev. Mat. Iberoamericana
  \textbf{10} (1994), no.~2, 283--347. \MR{1286477 (95e:42039)}

\bibitem{MR2399059}
S.~Meda, P.~Sj{\"o}gren, and M.~Vallarino, \emph{On the {$H^1$}-{$L^1$}
  boundedness of operators}, Proc. Amer. Math. Soc. \textbf{136} (2008), no.~8,
  2921--2931. \MR{2399059 (2009b:42025)}

\bibitem{MR0380394}
J.~Peetre, \emph{On spaces of {T}riebel-{L}izorkin type}, Ark. Mat. \textbf{13}
  (1975), 123--130. \MR{0380394 (52 \#1294)}

\bibitem{MR604370}
M.~Taibleson and G.~Weiss, \emph{The molecular characterization of certain
  {H}ardy spaces}, Representation theorems for {H}ardy spaces, Ast\'erisque,
  vol.~77, Soc. Math. France, Paris, 1980, pp.~67--149. \MR{604370 (83g:42012)}

\end{thebibliography}

\end{document}